\documentclass[12pt,reqno]{amsart}
\usepackage{amsmath}
\usepackage{amssymb, gensymb}
\usepackage[left=3cm,top=3cm,right=3cm,bottom=3cm]{geometry}
\usepackage{epsfig}
\usepackage{graphicx}
\usepackage[usenames,dvipsnames]{color}

\begin{document}
\newtheorem{theorem}{Theorem}[section]
\newtheorem{lemma}[theorem]{Lemma}
\newtheorem{definition}[theorem]{Definition}
\newtheorem{conjecture}[theorem]{Conjecture}
\newtheorem{proposition}[theorem]{Proposition}
\newtheorem{algorithm}[theorem]{Algorithm}
\newtheorem{corollary}[theorem]{Corollary}
\newtheorem{observation}[theorem]{Observation}
\newtheorem{problem}[theorem]{Open Problem}
\newtheorem{remark}[theorem]{Remark}
\newcommand{\noin}{\noindent}
\newcommand{\ind}{\indent}
\newcommand{\om}{\omega}
\newcommand{\I}{\mathcal I}
\newcommand{\N}{{\mathbb N}}
\newcommand{\LL}{\mathbb{L}}
\newcommand{\R}{{\mathbb R}}
\newcommand{\E}{\mathbb E}
\newcommand{\Prob}{\mathbb{P}}
\newcommand{\pr}{\mathbb{P}}
\newcommand{\eps}{\varepsilon}
\newcommand{\G}{{\mathcal{G}}}
\newcommand{\Bin}{\mathrm{Bin}}

\newcommand{\cD}{{\mathcal D}}
\newcommand{\cE}{{\mathcal E}}
\newcommand{\cF}{{\mathcal F}}
\newcommand{\cG}{{\mathcal G}}
\newcommand{\cH}{{\mathcal H}}
\newcommand{\cR}{{\mathcal R}}

\newcommand{\RG} {\ensuremath{\mathcal G(n,r)}}
\newcommand{\RGuv} {\ensuremath{\widetilde{\mathcal G}(n,r)}}
\newcommand{\SR} {\ensuremath{\mathcal S_n}}
\newcommand{\SRn} {\ensuremath{\mathcal S_n}}
\newcommand{\MCF}{{\rm maximal-clique-free}}
\newcommand{\mcf}{{\rm mcf}}
\newcommand{\whp}{{\rm whp}}

\title{Clique colouring of binomial random graphs}

\author{Colin McDiarmid}
\address{Department of Statistics, Oxford University, 24 - 29 St Giles, Oxford OX1 3LB, UK}
\email{\tt cmcd@stats.ox.ac.uk}

\author{Dieter Mitsche}
\address{Universit\'{e} de Nice Sophia-Antipolis, Laboratoire J-A Dieudonn\'{e}, Parc Valrose, 06108 Nice cedex 02}
\email{\texttt{dmitsche@unice.fr}}

\author{Pawe\l{} Pra\l{}at}
\address{Department of Mathematics, Ryerson University, Toronto, ON, Canada
\and
The Fields Institute for Research in Mathematical Sciences, Toronto, ON, Canada}
\email{\tt pralat@ryerson.ca}

\thanks{The research of the third author is supported in part by NSERC and Ryerson University.}
\date{\today}
\maketitle

\begin{abstract}
A {clique colouring} of a graph is a colouring of the vertices so that no maximal clique is monochromatic (ignoring isolated vertices). The smallest number of colours in such a colouring is the {clique chromatic number}.
In this paper, we study the asymptotic behaviour of the clique chromatic number of the random graph $\G(n,p)$ for a wide range of edge-probabilities $p=p(n)$. We see that the typical clique chromatic number, as a function of the average degree, forms an intriguing step function. 
\end{abstract}

\section{Introduction and main results}\label{sec:intro}

A \emph{proper colouring} of a graph is a labeling of its vertices with colours such that no two vertices sharing the same edge have the same colour.  The smallest number of colours in a proper colouring of a graph $G=(V,E)$ is called its \emph{chromatic number}, and it is denoted by $\chi(G)$. 

In this paper we are concerned with another notion of vertex colouring. A \emph{clique} $S \subseteq V$ is a subset of the vertex set such that any pair of vertices in $S$ is connected by an edge. Moreover, a clique $S$ is \emph{maximal} if there is no vertex in $V \setminus S$ connected by an edge  to every vertex in $S$ (in other words, $S$ is not a proper subset of another clique). A \emph{clique colouring} of a graph $G$ is a colouring of the vertices so that no maximal clique is monochromatic, ignoring isolated vertices. The smallest number of colours in such a colouring is called the \emph{clique chromatic number} of $G$, denoted by $\chi_c(G)$. Clearly, $\chi_c(G) \le \chi(G)$ but it is possible that $\chi_c(G)$ is much smaller than $\chi(G)$. For example, for any $n \geq 2$ we have $\chi(K_n) = n$ but $\chi_c(K_n) = 2$.  Note that if $G$ is triangle-free then $\chi_c(G)=\chi(G)$.

\smallskip

The problem has received considerable attention for deterministic graphs: in~\cite{Mohar} it was shown that planar graphs satisfy $\chi_c(G) \leq 3$. In~\cite{Andreae} a necessary and sufficient condition for $\chi_c(G) \le k$ on line graphs was given. Moreover, several graph classes are known to satisfy $\chi_c(G) \leq 2$:  claw-free perfect graphs~\cite{BGGPS2004}, co-diamond free graphs~\cite{Defossez}, claw-free planar graphs~\cite{SLK2014}, powers of cycles (other than odd cycles longer than three that need three colours)~\cite{Campos}, and also claw-free graphs with maximum degree at most $7$ (except for odd cycles longer than $3$)~\cite{Liang}. Also, circular-arc graphs are known to have $\chi_c(G) \le 3$ (see~\cite{Cerioli}). Further results about other classes of graphs can also be found in~\cite{Klein}, and the clique chromatic number of graphs without having long paths was studied in~\cite{Gravier}. On the algorithmic side, it is known that testing whether $\chi_c(G)=2$ for a planar graph can be performed in polynomial time~\cite{Kratochvil}, but deciding whether $\chi_c(G)=2$ is $NP$-complete for $3$-chromatic perfect graphs~\cite{Kratochvil} and for graphs with maximum degree $3$~\cite{BGGPS2004}.  The clique chromatic number for geometric graphs (in particular, random geometric graphs) is analysed in the accompanying paper~\cite{OurpaperGnr}.

\smallskip

Let us recall the classic model of random graphs that we study in this paper. The \emph{binomial random graph} $\G(n,p)$ is the random graph $G$ with vertex set $[n]$ in which every pair $\{i,j\} \in \binom{[n]}{2}$ appears independently as an edge in $G$ with probability~$p$. Note that $p=p(n)$ may (and usually does) tend to zero as $n$ tends to infinity.  The behaviour of many colouring problems has been investigated for $\G(n,p)$: the classic chromatic number has been intensively studied,  see~\cite{JLR, KM15} and the references therein; the list chromatic number (known also as the choice number) was studied among others in~\cite{KSVW02, Vu1}, and other variants were analysed recently in~\cite{DMP, DMNPPT} and~\cite{FMPP}.
\smallskip

All asymptotics throughout are as $n \rightarrow \infty $ (we emphasize that the notations $o(\cdot)$ and $O(\cdot)$ refer to functions of $n$, not necessarily positive, whose growth is bounded). We use the notations $f \ll g$ for $f=o(g)$ and $f \gg g$ for $g=o(f)$. 
{We also write $f(n) \sim g(n)$ if $f(n)/g(n) \to 1$ as $n \to \infty$ (that is, when $f(n) = (1+o(1)) g(n)$). }
We say that events $A_n$ in a probability space hold \emph{with high probability} (or \emph{\whp}), if the probability that $A_n$ holds tends to $1$ as $n$ goes to infinity. Since we aim for results that hold \whp, we will always assume that $n$ is large enough. We often write $\G(n,p)$ when we mean a graph drawn from the distribution $\G(n,p)$.  
Finally, we use $\log n$ to denote natural logarithms.
\smallskip

Here is our main result.  We consider the edge probability $p(n)$ ranging from the sparse case when $pn \to \infty$ arbitrarily slowly, to the dense case when $p = 1-\eps$ for an arbitrarily small $\eps>0$, and we break this range into 8 parts.  After the theorem we give a corollary which is less precise but easier to read.
\begin{theorem}\label{thm:Gnp}
Let $\eps > 0$ be a constant (arbitrarily small). Let $\omega = \omega(n)$ be a function tending to infinity with $n$ (arbitrarily slowly), and suppose that $\omega = o(\sqrt{\log n})$.  Let $G \in \G(n,p)$ for some $p=p(n)$. Then, the following holds \whp:
\begin{itemize}
\item [(a)] If $pn \ge \omega$ and  $pn < n^{1/2-\omega/\sqrt{\log n}}$, then $\chi_c(G) \sim \chi(G) \sim \frac {pn}{2 \log(pn)}.$
\item [(b)] If $n^{1/2-\omega/\sqrt{\log n}} \le pn < \sqrt{2 n \log n}$, then $\chi_c(G) = \Omega \left( \frac {p^{3/2} n}{(\log n)^{1/2}} \right)$ and $\chi_c(G) \le \chi(G) \sim \frac {pn}{2 \log(pn)}.$
\item [(c)] If $\sqrt{2 n \log n} \le pn < n^{3/5 - (6/\log n)^{1/2}}$, then $\chi_c(G) = \Theta \left( \frac {p^{3/2} n}{(\log n)^{1/2}} \right).$ 
\item [(d)] If $n^{3/5 -  (6/\log n)^{1/2}} \le pn  < n^{3/5} (\log n)^{3/5}$, then $\chi_c(G) = n^{2/5+o(1)}.$ 
\item [(e)] If $n^{3/5} (\log n)^{3/5} \le pn < n^{2/3-\eps}$, then $\chi_c(G) = \Omega \left( 1/ p \right)$ and $\chi_c(G) = O(\log n / p).$
\item [(f)] If $pn = n^{2/3+o(1)}$ and $pn < n^{2/3} (\log n)^{4/9}$, then $\chi_c(G) = n^{1/3+o(1)}.$
\item [(g)] If $n^{2/3}(\log n)^{4/9} \le pn < n^{1-\eps}$, then $\chi_c(G) = \Omega \left( 1/ p \right)$ and $\chi_c(G) = O(\log n / p).$
\item [(h)] If $p = n^{-o(1)}$ and $p \le 1-\eps$, then $\chi_c(G) \le (1/2+o(1)) \log_{1/(1-p)} n.$
\end{itemize}
\end{theorem}

\noindent Note that 
$$
\log_{1/(1-p)} n =
\begin{cases}
\Theta(\log n) & \text{ if } p = \Omega(1) \text{ and } p \le 1\!-\!\eps \text{ for some $\eps\!>\!0$,} \\
(1+o(1)) (\log n) / p & \text{ if } p = o(1).
\end{cases}
$$
In places a slightly tighter result than the one given in Theorem~\ref{thm:Gnp} can be obtained. 
We are interested in the ratio between the upper and lower bounds on $\chi_c(G)$ for each $p$, but in order to keep the statement reasonably simple, we spelled out results only for large intervals of $p$. For instance, as cases (d) and (f) are concerned with $pn = n^{3/5+o(1)}$ and $pn = n^{2/3+o(1)}$ respectively, we treated these cases in the statement of Theorem~\ref{thm:Gnp} less fully. For the reader interested in more precise bounds on this ratio, we refer to the proofs in the sections below for more details. 

\smallskip
In order to understand better the behaviour of the clique chromatic number, let us introduce the following function $f : (0,1) \to \R$:
$$
f(x) = 
\begin{cases}
x & 0 < x < 1/2\\
1+3(x-1)/2 & 1/2 \le x < 3/5\\
1-x & 3/5 \le x < 1.
\end{cases}
$$
This function is depicted in Figure~\ref{fig1}. We get immediately the following corollary.

\begin{figure}[h]
\begin{center}
\includegraphics[width=2.0in,height=1.6in]{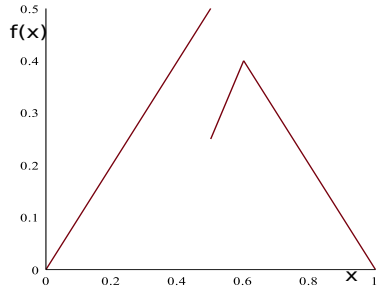}
\end{center}
\caption{The function $f(x)$ related to the clique chromatic number of $G(n,p)$}\label{fig1}
\end{figure}

\begin{corollary} \label{cor.main}
Let $f : (0,1) \to \R$ be defined as above. Let $G \in \G(n,p)$ for some $p=p(n)$. If $pn = n^{x+o(1)}$ for some $x \in (0,1/2) \cup (1/2, 1),$
then \whp\  $\chi_c(G) = n^{f(x)+o(1)}$.
\end{corollary}

Figure~\ref{fig1} shows that there are dramatic transitions in Corollary~\ref{cor.main} when $p$ is about $n^{-1/2}$ and $n^{-2/5}$, and there is some problem at about $n^{-1/3}$ in Theorem~\ref{thm:Gnp}.  What is happening at these points?  Here is the rough story.

For $p \ll n^{-1/2}$,
\whp\ most edges are not in triangles, and $\chi_c(G)$ is close to $\chi(G)$.
For $p \gg n^{-1/2}$, \whp\ each edge is in many triangles, and we may use triangle-free sets as colour classes; or we may find relatively small dominating sets which yield an upper bound on $\chi_c(G)$ (see Lemma~\ref{lem:deterministic_upperbound}). For $p \ll n^{-2/5}$, the first bound (which increases with $p$) is stronger, whereas from then on the second bound (which decreases with $p$) gives better results.
For lower bounds on $\chi_c(G)$, depending on the range of $p$, we find a value of $k$, so that \whp\ there are many $k$-cliques and most of them are not contained inside $(k+1)$-cliques.
(In the central range in Figure~\ref{fig1}, for $n^{-1/2} \ll p \ll n^{-2/5}$, we use $k=3$.)
For $p \ll n^{-1/3}$ \whp\ most triangles are not contained in $4$-cliques; and for $p \gg n^{-1/3}$ \whp\ each triangle is in many $4$-cliques.



\section{Preliminaries}

In this section, let us recall a few inequalities that are well known and can be found, for example, in~\cite{JLR}. We will use the following version of \emph{Chernoff's bound}. Suppose that $X \in \Bin(n,p)$ is a binomial random variable with expectation $\mu=pn$. If $0<\delta<1$, then 
\begin{equation}\label{eq:chern1}
\Prob [X < (1-\delta)\mu] \le \exp \left( -\frac{\delta^2 \mu}{2} \right),
\end{equation}
and if $\delta > 0$,
\begin{equation}\label{eq:chern2}
\Prob [ X > (1+\delta)\mu] \le \exp\left(-\frac{\delta^2 \mu}{2+\delta}\right).
\end{equation}
\bigskip

We will use the following version of \emph{Janson's inequality}, adapted to our setting.

\begin{lemma}\label{thm:Janson}
Let $G =(V,E) \in \G(n,p)$, $S \subseteq V$, and $k \in \mathbb{N} \setminus \{1,2\}$. For $A \subseteq S$ with $|A|=k$, let $\mathcal{I}_A$ be the indicator variable which is $1$ if the vertices in $A$ form a clique of size $k$, and $0$ otherwise. Let $X=\sum_{A \subseteq S, |A|=k} \mathcal{I}_A$ be the random variable counting copies of $K_k$ in $S$, $\mu=\E [X]$, and $\overline{\Delta}= \sum_{A,B \subseteq S, |A|=|B|=k, |A \cap B| \geq 2} \E [\mathcal{I}_A \mathcal{I}_B ].$ Then,  for $0 \leq t \leq \E X$,
$$
\Prob \left( X \leq \mu - t \right) \le \exp\left(- \frac{\varphi(-t/\mu)\mu^2}{\overline{\Delta}}\right),
$$
where $\varphi(x)=(1+x)\log(1+x)-x$.
\end{lemma}

\noindent (Let us note that indicator random variables $\mathcal{I}_A$ and $\mathcal{I}_B$ are independent if $|A \cap B| \le 1$; see the definition of $\overline{\Delta}$ above.)
\bigskip

We will use the following result of Vu to obtain bounds on the upper tail of certain subgraphs (see~\cite{Vu}).
Denote by $e(H)$ ($v(H)$, respectively) the number of edges (vertices, respectively) of a graph $H$. We say that a graph $H$ is \emph{balanced} if for every subgraph $H' \subseteq H$ with $v(H') \geq 1$, we have $e(H')/v(H') \leq e(H)/v(H)$.
\begin{lemma}\label{lem:vu}
Let $H$ be a fixed balanced graph on $k$ vertices, and let $K>0$ be a constant.  Then there are constants $c=c(H,K)>0$ and $s=s(H,K)$ such that the following holds. Let $G \in \mathcal{G}(n,p)$, let $Y$ denote the number of appearances of $H$ in the graph $G$, and let $\mu= \E [Y]$.
If $0 < \varepsilon \leq K$ and $\eps^2 \mu^{1/(k-1)} \geq s \log n$, then
$$
\mathbb{P}(Y \geq (1+\varepsilon)\mu) \leq \exp(-c\varepsilon^2 \mu^{1/(k-1)}).
$$
\end{lemma}

\noindent (Let us remark that stronger bounds on the upper tail are known; see, for example,~\cite{DeMarcoKahn}. For our purposes, however, the current one is sufficient.)


\bigskip

We will also make use of the following version of \emph{Harris's inequality} (see~\cite{Harris}). 
Let $G \in \G(n,p)$, with vertex set $V=[n]$ and random edge set $E$.  Let $F \subseteq \binom{V}{2}$ so $F$ is a set of possible edges. Let $\Omega$ be the collection of all subsets of the set $F$.  For $A \subseteq \Omega$ let $\Prob(A)$ mean $\Prob(E \cap F \in A)$.
For $A \subseteq \Omega$, we say that $A$ is a \emph{downset} if $x \in A$ and $y \subseteq x$ implies $y \in A$, and $A$ is an \emph{upset} if $x \in A$ and $x \subseteq y \subseteq \Omega$ implies $y \in A$.  We also may refer to the event $E \cap F \in A$ as a \emph{down-event} or an \emph{up-event} in these two cases.
\begin{lemma}\label{lem.Harris}
With notation as above, let $A$ and $B$ be upsets and let $C$ and $D$ be downsets.
Then $A$ and $B$ are positively correlated, i.e., $\Prob (A \cap B) \geq \Prob (A) \Prob (B)$, and similarly $C$ and $D$ are positively correlated;
and $A$ and $C$ are negatively correlated, i.e., $\Prob (A \cap C) \leq \Prob (A) \Prob (C)$.
\end{lemma}

The final lemma here is a simplified version of Lemma~2.3 of~\cite{LMU03}.  It says roughly that a random walk which is biased upwards when needed tends to move upwards as we should expect.  We shall use it in the proof of Theorem~\ref{thm.ubforc} (to analyse a greedy procedure to find triangle-free sets).

\begin{lemma} \label{lem.bin2}
Let $\cF_0 \subseteq \cF_1 \subseteq \ldots \subseteq \cF_{k}$ be a filter
(in a probability space).
Let $Y_1,Y_2,\ldots,Y_k$ be binary random variables such that
each $Y_i$ is $\cF_i$-measurable, and let $Z = \sum_{i=1}^{k} Y_i$.
Let $E_0,E_1,\ldots,E_{k-1}$ be events where $E_i \in \cF_i$ for $i=0,\ldots,k-1$.
Let $0 \leq p \leq 1$ and let $z$ be a positive integer.
Suppose that for each $i= 1,\ldots,k$
\[ \Prob (Y_i = 1 \mid \cF_{i-1}) \geq p \;\;\; \mbox{ on } \; E_{i-1}. \]
Then  
\[ \Prob \left( (Z < z) \land \left(\bigwedge_i E_i \right) \right) \leq \pr (\Bin(k,p) < z).\]
\end{lemma}


\section{Upper bounds}

\subsection{Upper bounds for parts (a) and (b)}

Part (a) and (b) of Theorem~\ref{thm:Gnp} follow immediately from the fact that $\chi_c(G) \le \chi(G)$ for each graph. It is well known that \whp
$$
\chi(\G(n,p)) \sim \frac {pn}{2 \log (pn)},
$$
provided that $pn \to \infty$ as $n \to \infty$, and $p=o(1)$ (see~\cite{L,M}). 

\subsection{Upper bound for part (c)}
For $p$ as here, we shall see that \whp\ each edge is in a triangle, so we can use triangle-free sets of vertices as colour classes.  We repeatedly use a natural greedy algorithm to construct maximal triangle-free sets and remove them.  While many vertices remain, with high probability each set constructed has size at least about $p^{-3/2} (\log n)^{1/2}$, so we will not remove too many of them. Later phases require a negligible number of further colours.  In the proof, in order to show that our greedy triangle-free sets are large, we need first to show that they do not tend to contain more edges than a typical set of the same size.

\begin{theorem} \label{thm.ubforc}
Suppose that $p = p(n)$ is such that $\sqrt{2 \log n /n} \leq p \leq 1/\log n$.  Let $c > \sqrt{2}$ be a constant. Then \whp\ for $G \in \G(n,p)$ we have
\[ \chi_c(G) \leq  c \, p^{3/2}n \, (\log n)^{-1/2}.\] 
\end{theorem}

\begin{proof}
First let us check that \whp\ each edge is in a triangle, for then we can colour using triangle-free sets.
The expected number of edges that are not part of any triangle is
\[ \binom{n}{2} p (1-p^2)^{n-2} \leq n^2 p e^{-np^2} (1-p^2)^{-2} = o(1)\] 
since $e^{-np^2} \leq n^{-2}$ and $p=o(1)$. The desired property now follows from Markov's inequality.

We will repeatedly use the following greedy algorithm to construct a
triangle-free set $A$ of vertices in a graph $G$ with vertex set a subset of the positive integers.  Consider the vertices in increasing order.  Initially set $A=\emptyset$ and $W=V$.
While $W$ is non-empty: remove the first member $w$ from $W$, and if $A \cup \{w\}$ is triangle-free then add $w$ to $A$.  Note that this procedure examines only potential edges with at least one end in $A$.

How large is the final triangle-free set $A$ for $G \in \G(n,p)$?  The first thing to consider is the number of edges of $G$ within $A$ when $|A|=a$.  We shall see in part (1) of the proof below that this is stochastically at most $\Bin(\binom{a}{2},p)$.  In part (2) of the proof we shall define appropriate integers $j_0$ and $t$ such that \whp\ we find a triangle-free set of size $j_0$ within $t$ steps; and we complete the proof in part (3).

\emph{Part (1)}. First we consider the number of edges within the set $A$.
Let $n$ be a (fixed) integer, with corresponding $p=p(n)$, and embed the random graph $G$ in an infinite random graph $G'$ on $1,2,3,\ldots$ with the same edge probability $p$.  Let $E'$ be the random edge-set of $G'$.  For $j=1,2,\ldots$ let $A_j$ be the set $A$ when $|A|=j$ (which is well defined for each $j$ since we are working in $G'$ and $p<1$).  Let $X_j$ be the number of edges within $A_j$ (and set $X_0=0$), and let $Y_j= X_j-X_{j-1}$.
 
We {\bf claim} that $X_j \leq_s \Bin(\binom{j}{2},p)$.  This is clearly true for $j=1$ since $X_1=0$.  Let $j \geq 2$  and suppose it is true for $j-1$.  
Let $0 \leq x \leq \binom{j-1}{2}$ and let $E_1$ be the event that $X_{j-1}=x$.  Let $B$ be a set of $j-1$ vertices, 
let $F_0$ be a set of $x$ edges within $B$, and let $w$ be an element of $V$ after all of $B$.
Let $E_2$ be the event that $A_{j-1}= B$ and $F_0$ is the set of edges of $G'$ within $B$.
Let $E_3$ be the event that $w$ is the vertex added to $A_{j-1}$ to form $A_j$.

Let $F_1$ be the set of $j-1$ possible edges between $w$ and $B$.  
Let $F'_1 = E' \cap F_1$, the random set of edges present in $F_1$.  Note that $|F'_1| \sim \Bin(j-1,p)$.  For each real $t$, let $f(t) = \Prob (\Bin(j-1,p) \geq t)$: and observe that $|F'_1| \geq t$ is an up-event.
Let 
\[\cF_1=\{F \subseteq F_1 : |F \cap \{ uw, vw \}| \leq 1 \mbox{ for each edge } uv \in F_0 \},\]
and observe that $F'_1 \in \cF_1$ is a down-event. 
Hence, for each $t$, by Harris's inequality, Lemma~\ref{lem.Harris},
\begin{eqnarray*}
   \Prob (Y_j \geq t \mid E_1 \land E_2 \land E_3) 
 & = &
  \Prob \left( |F'_1| \geq t \mid E_1 \land E_2 \land E_3 \right) \\
 & = &
  \Prob \left( |F'_1| \geq t  \mid F'_1 \in \cF_1 \right)  \\
 & \leq & 
  \Prob \left( |F'_1| \geq t \right) \; = \; f(t).
\end{eqnarray*}
Since this holds for each possible choice of $B$, $F_0$ and $w$ specifying events $E_2$ and $E_3$, we obtain
 \[ \Prob(Y_j \geq t \mid X_{j-1}=x) = \Prob(Y_j \geq t \mid E_1) \leq f(t).\]
Now fix $t \geq 0$, and let $g(x)= f(t-x) =\Prob(\Bin(j-1,p) \geq t-x)$. 
By the last inequality, since $X_j=X_{j-1}+Y_j$,
\begin{eqnarray*}
   \Prob(X_j \geq t) &=&
   \sum_x \Prob(X_{j-1}=x) \, \Prob(Y_j \geq t-x \mid X_{j-1}=x)\\
 & \leq &
   \sum_{x} \Prob(X_{j-1}=x) \, g(x) \; = \E[g(X_{j-1})].
\end{eqnarray*}
But $X_{j-1} \leq_s \Bin(\binom{j-1}{2},p)$ by the induction hypothesis, and $g(x)$ is non-decreasing, so
\begin{eqnarray*}
   \E[g(X_{j-1})] & \leq &
   \E[g(\Bin(\binom{j\!-\!1}{2},p))] \\
 &=&
   \sum_x \Prob(\Bin(\binom{j\!-\!1}{2},p)=x) \, \Prob(\Bin(j\!-\!1,p) \geq t\!-\!x)\\
  &=& 
   \Prob(\Bin(\binom{j}{2},p) \geq t).
\end{eqnarray*}
Thus 
\[ \Prob(X_j \geq t) \leq \Prob(\Bin(\binom{j}{2},p) \geq t) , \]
and we have established the claim that $X_j \leq_s \Bin(\binom{j}{2},p)$ for each $j$.  This completes part (1) of the proof.
\smallskip
 
\emph{Part (2)}. Still in the infinite case, let $T_j$ be the number of vertices $w$ tested until we find one to add to $A_{j-1}$.  For $j=1, 2$ the first available vertex must be added. Fix $j \geq 3$.  For each positive integer $i$, let $W^i$ be the $i$th vertex in $W$ after the largest vertex in $A_{j-1}$ (the $i$th candidate to add to $A_{j-1}$).
 
Let $x$, $B$, $F$, $w$, $E_1$, $E_2$ and $\cF_1$ be as before.  Observe that $\cF_1$ is an intersection of $x$ down-sets in $2^{F_1}$, and for each of these downsets $\cD$ we have $\pr(F'_1 \in \cD)= (1-p^2)$.  By Harris's inequality (Lemma~\ref{lem.Harris}) again, 
\[ \pr(w \mbox{ can be added to } A_{j-1} \mid E_1 \land E_2)
= \pr(F'_1 \in \cF_1) \geq (1-p^2)^x.\]

Since this holds for each possible choice of $B$, $F$ and $w$, it follows that for each positive integer $i$ we have
\[ \pr(W^i \mbox{ can be added to } A_{j-1} \mid X_{j-1} = x)  \geq (1-p^2)^x. \]
Hence
\[ \pr(T_j \geq t \mid X_{j-1} =x) \leq (1-(1-p^2)^x)^{t-1},\]
and since the upper bound is nondecreasing in $x$, it follows that
\[ \pr(T_j \geq t \mid X_{j-1} \leq x) \leq (1-(1-p^2)^x)^{t-1}.\]
Let $0< \eta < 1/\sqrt{2}$, and let $j_0 \sim \eta p^{-3/2} (\log n)^{1/2}$.  Note that $j_0 = O(n^{3/4} (\log n)^{-1/4}) = o(n^{3/4})$, and $j_0 = \Omega(\log^{2}n)$. Observe that $\frac12 \eta^2 < \frac14$; and let $\frac12 \eta^2 < \alpha < \beta < \frac14$.
Let $\mu_0 = \binom{j_0}{2} p \sim (\eta^2 /2) p^{-2} \log n$; and let $x_0 \sim \alpha p^{-2} \log n$.  Then for each $j =1,\ldots, j_0$, by the Chernoff bound~(\ref{eq:chern2}) 
\[ \Prob \left( \Bin \left( \binom{j\!-\!1}{2},p \right) > x_0 \right) \leq \Prob \left( \Bin \left( \binom{j_0}{2},p \right) > x_0 \right) \leq e^{- \Omega(\mu_0)} =o(n^{-2}),\]
where the final step follows since $p=o(1)$.
Let $t= p^{-\frac32} n^{\beta}$, so $t = o(n^{\frac34 + \beta})$ and $t/j_0 = n^{\beta+o(1)}$.  Then
\begin{eqnarray*}
    \pr \left( \sum_{j=1}^{j_0} T_j \geq t \right) 
 & \leq & 
  \sum_{j=1}^{j_0} \pr(T_j \geq t/j_0 ) \\
  & \leq & 
  \sum_{j=1}^{j_0} \Big( \pr(T_j \geq t/j_0 \mid X_{j-1} \leq x_0) + \Prob(X_{j-1} > x_0) \Big)\\
  & \leq & j_0 (1- (1-p^2)^{x_0})^{t/j_0 -1} +  j_0 \, \Prob(X_{j_0 -1} > x_0)  \\
  & \leq &  
   n (1- (1-p^2)^{x_0})^{t/j_0 -1} + o(1/n)
\end{eqnarray*}
by the above. But since $p=o(1)$, 
\[(1-p^2)^{x_0} =  e^{-(1+o(1))p^2 x_0} = e^{-(\alpha +o(1)) \log n} = n^{ - \alpha +o(1)} \]
and so 
\[ (1- (1-p^2)^{x_0})^{t/j_0 -1} \leq \exp(- n^{ - \alpha +o(1)} \cdot n^{\beta+o(1)}) = \exp(- n^{ \beta - \alpha +o(1)}). \]
Thus
\[ \pr \left( \sum_{j=1}^{j_0} T_j \geq t \right) \leq n \exp(- n^{ \beta - \alpha +o(1)}) + o(1/n) = o(1/n). \]
In other words, the probability that the greedy algorithm fails to find a triangle-free set of size $j_0$ within the first $t-1$ available vertices is $o(1/n)$.
Now let us return to the finite case $G \in \G(n,p)$, and note that the last statement still applies.  This completes part (2) of the proof.

\emph{Part (3)}. Let us repeatedly seek a maximal triangle-free set of size at least $j_0$, and remove it, as long as at least $t$ vertices remain.
Let $V'$ be the set of vertices remaining at some step (when we are about to seek a new triangle-free set).
Each time, we have $|V'| \geq t$ (that is, we start with at least $t$ vertices), and the potential edges have not been examined before.  Hence, each time, the probability we fail is $o(1/n)$, and so the probability we ever fail is $o(1)$.  This whole process uses at most $n/j_0 \sim \eta^{-1} p^{3/2} n (\log n)^{-1/2}$ colours.  Observe for later that this number is $\Omega(n^{1/4})$. 
\smallskip

By this stage, $|V'| <t$.  While $|V'| \geq \log^2 n$, we continue to seek maximal triangle-free sets using the same greedy algorithm, but now we will be content with smaller sets.  Any remaining vertices can each be given a new colour, adding at most $\log^2 n$ colours.  We need new notation.

Let $s^* = \lceil p^{-3/2}\rceil$, and let $s=\min\{s^*, |V'| \}$.  List the vertices in $V'$ in their natural increasing order as $v'_1, v'_2,\ldots$. For $i=1,\ldots,s$ let $\cF_i$ be the $\sigma$-field generated by the appearance or not of the edges amongst $v'_1,\ldots,v'_i$; let $E_i$ be the random set of edges amongst $v'_1,\ldots,v'_i$ which appear; and let $Y_i$ be the indicator that the $i$th vertex $v'_i$ is accepted into $A$.
%
Let $X=|E_{s}|$.  Then $X \sim \Bin(\binom{s}{2},p)$, with expected value at most about $\frac12 p^{-2} = \Omega(\log^2 n)$.  Let $x_0=\frac23 p^{-2}$.  Then $\Prob(X>x_0) = o(1/n)$ by the Chernoff bound~(\ref{eq:chern2}).

Given that $E_{i-1}=F$, the event that $Y_i=0$ is the union over the edges $e \in F$ of the event that $v'_i$ is adjacent to both end vertices of $e$.  Thus, for each event $B \in \cF_{i-1}$
\[ \pr\left(Y_i=0 \mid B \land (|E_{i-1}| \leq x_0)\right) \leq x_0 p^2 \leq 2/3\]
(assuming that the conditioning event has positive probability).
Hence
\[ \pr(Y_i=1 \mid \cF_{i-1}) \geq 1/3 \;\; \mbox{ on the event } (|E_{i-1}| \leq x_0).\]

Let $Z$ be the value of $|A|$ after trying to add $v'_1,\ldots,v'_s$ (starting from $A=\emptyset$).  Thus $Z= \sum_{i=1}^{s} Y_i$.  Hence by Lemma~\ref{lem.bin2} and the Chernoff bound~(\ref{eq:chern2})
\[ \pr((Z < s/4) \land (X_{s-1} \leq x_0)) \leq \pr( \Bin(s,1/3) < s/4) = e^{-\Omega(s)} = o(1/n), \]
and so
\[ \pr(Z < s/4) \leq \pr(X_{s-1} > x_0) +o(1/n) = o(1/n).\]
Throughout the time when $t>|V'| \geq s^*$, \whp\ each triangle-free set $A$ found has size at least $s^*/4$; so the number of colours used is at most $4t/s^* \sim 4n^{\beta}$.  Throughout the time when $s^* >|V'| \geq \log^2 n$, \whp\ each triangle-free set $A$ found has size at least $|V'|/4$; so the number of colours used is $O(\log |V'|) = O(\log n)$.   As we noted before, the final phase uses $O(\log^2 n)$ colours.  Thus the total number of colours used after the first phase is $O(n^{\beta})$, which is negligible compared with $n/j_0 = \Omega(n^{\frac14})$.  
\end{proof}

\subsection{Upper bounds for parts (d)-(g)}

Before we state a useful observation that holds for deterministic graphs, let us introduce a few definitions. An \emph{independent set} (or a \emph{stable set}) of a graph $G = (V, E)$ is a set of vertices in $G$, no two of which are adjacent. A \emph{dominating set} in $G$ is a subset $D$ of $V$ such that every vertex not in $D$ is adjacent to at least one member of $D$. Of course, a set could be both independent and dominating; this situation happens if and only if the set is a maximal independent set. Finally, for $v \in V$ let $N(v)$ denote the set of neighbours of $v$.

The following lemma is part of Theorem 3 of~\cite{BGGPS2004}: we give a short proof here for completeness.  
\begin{lemma}\label{lem:deterministic_upperbound}
Let $G = (V,E)$ be a graph, 
and suppose that $A \subseteq V$ is a dominating set. 
Then $\chi_c(G) \le |A|+1$.
\end{lemma}
\begin{proof}
Let $A = \{v_1, v_2, \ldots, v_k\}$.
For $i=1,\ldots,k$ in turn, assign colour $i$ to each uncoloured neighbour of $v_i$.
Any vertices remaining uncoloured must form an independent subset of $A$: give each such vertex colour 0.
This gives a proper clique-colouring, since any vertices coloured 0 form an independent set; and for each $i=1,\ldots,k$ the set of vertices coloured $i$ is dominated by $v_i$, which is not coloured $i$.
\end{proof}

From Lemma~\ref{lem:deterministic_upperbound}, we quickly get the following bound for binomial random graphs. This proves the upper bounds in parts (d) to (g) of Theorem~\ref{thm:Gnp}.
We say that $p=p(n)$ is \emph{bounded below 1} if there is some constant $\varepsilon>0$ such that $p \leq 1-\varepsilon$ for all $n$.

\begin{lemma}\label{lem.upperbound_sparse_gnp}
Suppose that $p=p(n)$ is bounded below 1, and $\omega = \omega(n) := pn - \log n \to \infty$ as $n \to \infty$.
Let $k=k(n) = \lceil \log_{1/(1-p)} n + \log_{1/(1-p)} \omega \rceil$. Then for $G \in \G(n,p)$ \whp 
$$
\chi_c(G) \; \le \; k + 1 \; \sim \; \log_{1/(1-p)} n.
$$
\end{lemma}
\begin{proof}
We construct a maximal independent set in the usual greedy way, by fixing an arbitrary order of the vertices $v_1,v_2,\ldots,v_n$ and adding vertices (in this order, one by one) to form a maximal independent set. Observe that at the beginning of a given step $t$ of this process, no edge emanating from $v_t$ has been exposed so far. Note that \whp\ the process terminates after at most $k$
vertices are added to the independent set: indeed, the probability that a larger set is created during this greedy process is 
the probability that a set of size $k$ is created at some point of the process, and this set does not dominate the remaining vertices; and (by considering just the second part) this probability is at most $n (1-p)^k \le \omega^{-1} = o(1)$. The result now follows by Lemma~\ref{lem:deterministic_upperbound}.
\end{proof}

\subsection{Upper bound for part (h)}

We will now show that for dense random graphs, we can improve the bound from part (d) by a factor of about $2$. This then proves part (h) of Theorem~\ref{thm:Gnp}. (For sparser graphs, the improvement might also be possible but the argument would be more tedious.)

\begin{theorem}
Suppose that $p=p(n)$ is bounded below 1, and satisfies $p = n^{-o(1)}$.
Then, for $G \in \G(n,p)$ \whp 
$$
\chi_c(G) \le (1/2+o(1)) \log_{1/(1-p)} n.
$$
\end{theorem}

\begin{proof}
First, let us start with a few simple properties of $\G(n,p)$.
\medskip 

\noindent \emph{Claim 1} Whp there is no clique of size $k = \lceil 2 \log_{1/p} n + 1 \rceil \sim 2 \log_{1/p} n$. 
\smallskip 

Indeed, noting that $k \gg 1$ we get that the expected number of cliques of size $k$ is
\begin{eqnarray*}
{ n \choose k } p^{k \choose 2} \le \left( \frac {ne}{k} p^{(k-1)/2} \right)^k \le \left( \frac {e}{k} \right)^k = o(1).
\end{eqnarray*}
The claim holds by Markov's inequality.
\medskip 

\noindent \emph{Claim 2} Whp there is no maximal clique of size at most $k = \lfloor \log_{1/p} n - 3 \log_{1/p} \log n \rfloor \sim \log_{1/p} n.$
\smallskip 

Indeed, the expected number of sets of size $k$ (not necessarily inducing cliques) for which there is no vertex adjacent to all of them is
\begin{eqnarray*}
    {n \choose k } (1-p^k)^{n-k} 
  & \leq & 
    \left( \frac {ne}{k} \right)^k \exp \left( -p^k (n-k) \right) \\
  &=& 
    O \left( \exp \left( k \log n - \log^3 n \right) \right) = o(1).
\end{eqnarray*}
As before, the claim holds by Markov's inequality.
\medskip 

The rest of the proof uses ideas similar to the ones used in the proofs of Lemmas~\ref{lem:deterministic_upperbound} and~\ref{lem.upperbound_sparse_gnp}. Let 
$$
\gamma = \frac { 2 \log_{1/p} \log n}{\log_{1/p} n}
$$
and note that $\gamma \to 0$ as $n \to \infty$. As before, we fix an arbitrary order of the vertices and construct an independent set $A$ greedily, but we stop once its size is equal to $k = \lceil (1/2+\gamma) \log_{1/(1-p)} n \rceil$. As before, we use $k+1$ colours to create colour classes that do not contain any maximal cliques.
Let $N$ be the random number of vertices not in $A$ and not dominated by $A$ (and so not coloured yet).  Then $N$ is stochastically at most $\Bin(n, (1-p)^k)$, and so by~(\ref{eq:chern1}) \whp\ $N \leq N_0$, where $N_0 = \lfloor 2 n(1-p)^k \rfloor \leq 2n^{\frac12 - \gamma}$. 
We will show that \whp\ one additional colour is enough to finish the job. We may assume that $N \leq N_0$.  As this part of the random graph is not exposed yet, this subgraph is distributed as a subgraph of $\mathcal{G}(N_0,p)$ and Claim~1 implies that \whp\ the maximum size of a clique in the additional colour is at most 
\begin{eqnarray*}
    \lceil 2 \log_{1/p} N_0 \rceil 
  & \le & 
    (1- 2 \gamma) \log_{1/p} n + 2\log_{1/p} 2 + 1 \\
  &=&
    \log_{1/p} n - 4 \log_{1/p} \log n + 2\log_{1/p} 2 + 1\\
   & < &
    \lfloor \log_{1/p} n - 3 \log_{1/p} \log n \rfloor.
\end{eqnarray*}
On the other hand, Claim~2 implies that \whp\ no clique of such size is maximal in $G$. The proofs of the upper bounds of Theorem~\ref{thm:Gnp} are finished.
\end{proof}


\section{Lower bounds}

Given a graph $G$, we say that a set $S$ of vertices is \emph{maximal-clique-free} if it contains no clique of size at least 2 which is a maximal clique in $G$.  We let $\mcf(G)$ denote the maximum size of a maximal-clique-free set of vertices.  Observe that if $G$ has $n$ vertices then
\[ \chi_c(G) \geq \frac{n}{\mcf(G)}.\]
Our lower bounds on $\chi_c(G)$ are usually obtained from upper bounds on $\mcf(G)$ by using this inequality.

\subsection{Lower bound for part (a)}

Let us first consider very sparse graphs.  Observe that for any graph $H$ with $n$ vertices and at most $t$ vertices in triangles, we have $\chi_c(H) \geq \frac{n-t}{\alpha(H)}$.  Let $G \in \G(n,p)$. If $pn \to \infty$ and $p=o(1)$, then $\chi(G) \sim \frac{n}{\alpha(G)} \sim \frac{pn}{2\log (pn)}$ \whp\ (see, for example,~\cite{L,M}).  For  $pn = o(n^{1/3})$ the expected number of triangles is $\binom{n}{3} p^3=o(n)$, so \whp\ the number of vertices in triangles is $o(n)$.  Thus for $pn\to \infty$ with $pn = o(n^{1/3})$
\begin{equation} \label{eqn.sparselb}
\chi_c(G) \geq (1+o(1)) \frac{pn}{2\log (pn)} \mbox{ \whp},
\end{equation}
and for such values of $pn$ the lower bound of part (a) is proven.

For the remaining range of $pn$ to be considered in part (a) we may therefore clearly assume $pn > \log n + \omega$. (In fact, we may assume $p \geq n^{1/3}/\omega)$ for some function $\omega=\omega(n) \to \infty$ as $n \to \infty$ but it would not simplify the argument.) 
Let $G \in \G(n,p)$, $\omega=\omega(n)$ be a function tending to infinity with $n$ arbitrarily slowly, and let
$$
c = c(n) = \frac{9 \log n}{\log (e/(np^2))}.
$$
(Observe that $9 \leq c \leq 9 \log n$.)

\smallskip

We will start with the following elementary observation. 
\begin{lemma}\label{lem:triangles}
Let $\omega, c$ be defined as above, and let  $\log n + \omega < pn \le \sqrt{n}$. Let $G \in \G(n,p)$. 
Then \whp\ no edge of $G$ belongs to more than $c$ triangles.
\end{lemma}
\begin{proof}
Consider a pair $u,v$ of distinct vertices. The number of common neighbours of $u$ and $v$ is a random variable $X$ with distribution $\Bin(n-2,p^2)$, which is stochastically bounded from above by $\Bin(n,p^2)$. 
We first observe that 
$$
\Prob( X \geq c ) \le {n \choose \lceil c \rceil} \left( p^2 \right)^{\lceil c \rceil} 
\le \left( \frac{en}{c} \cdot p^2 \right)^{\lceil c \rceil}  
\leq \left( \frac {np^2}{e} \right)^{c},
$$
since $c \geq 9 \geq e^2$ and $np^2/e \leq 1/e \leq 1$.
Hence,
$$
\Prob( X \ge c ) \le \exp \left( - c \log \left( \frac {e}{np^2} \right) \right) = n^{-9} = o(n^{-2}).
$$
Since there are $O(n^2)$ pairs to consider, the result holds by a union bound. 
\end{proof}

The next lemma is less straightforward. 

\begin{lemma}\label{lem:lower_bound1}
Let $\omega, c$ be defined as above. Suppose that $\log n + \omega <  pn \le n^{1/2-\omega/\sqrt{\log n}}$. Let 
$$
\eps = \eps(n) = 3 \left(2ecnp^2 \right)^{1/(2c)}
$$
and $\eps' = \eps'(n) = \max \{ \eps, 1/\log \log (pn) \}$.
Then $\eps'=o(1).$ Moreover, let $G=(V,E) \in \G(n,p)$. Then, \whp\ every set $K \subseteq V$ of size 
$$
k = (2+\eps') \frac {\log (pn)}{p} \sim \frac {2 \log(pn)}{p}
$$
has the following properties:
\begin{itemize}
\item [(a)] the graph induced by $K$, $G[K]$, has at least one edge, 
\item [(b)] there is no set of edge-disjoint triangles, each with one vertex in $V \setminus K$, which contains at least a $1/(2c)$-fraction of the edges in $G[K]$. 
\end{itemize}
\end{lemma}
\begin{proof}
First, recall that $c = (9 \log n)/(\log (e/(np^2))$ and so
$$
\eps = \eps(n) = 3 \left(2ecnp^2\right)^{1/(2c)} = 3 \exp \left( - \frac {\log(e/(np^2))}{18 \log n} \log \left( \frac {n \log (e/(np^2))}{18 e \, np^2 \, \log n} \right) \right)
$$
is an increasing function in $p$ (if we think of $p$ as a variable). Hence, writing $\tilde{\eps}$ as the value of $\eps$ when $p$ attains its maximum value $n^{-\frac12 - \omega/\sqrt{\log n}}$,
\begin{eqnarray*}
0\; < \; \eps \; \le \; \tilde{\eps}
&=& 3 \exp \left( - (1-o(1)) \frac {\omega}{9 \sqrt{\log n}} ( 2 \omega \sqrt{\log n}- O(\log \log n)) \right) \\
&=& 3 \exp ( - (2/9-o(1)) \omega^2 ) \; = \; o(1),
\end{eqnarray*}
and so also $\eps' = o(1)$, and the first assertion follows. On the other hand, it follows immediately from the definition of $\eps'$ that $\eps' \ge 1/\log \log(pn) \gg 1/\log(pn)$.

For the second assertion, part (a) is standard. The probability $p_0$ that the desired property is not satisfied can be estimated as follows: 
\begin{eqnarray*}
p_0 &\le& {n \choose k}  (1-p)^{k \choose 2} \le \left( \frac {en}{k} \right)^k \exp \left( - p { k \choose 2} \right) \\
&=& \exp \left( k \left( \log (en/k) - pk/2 + p/2 \right) \right) \\
&\le& \exp \left( k \left( \log (pn) + O(1) - \log \log (pn) - \frac {2+\eps'}{2} \log(pn) \right) \right) \\
& \le & \exp (-k) = o(1),
\end{eqnarray*}
where the second last step holds for $n$ sufficiently large.

For part (b), for a given $m$ such that $1 \le m \le {k \choose 2} = O(n^2)$, we are going to estimate the probability $p_m$ that there exists a set $K$ of size $k$ with precisely $m$ edges in $G[K]$ and with at least $\lceil m/(2c) \rceil$ edges that belong to edge-disjoint triangles with the third vertex in $V \setminus K$. We have 
$$
p_m \le {n \choose k} { {k \choose 2} \choose m} p^m (1-p)^{ {k \choose 2} - m} {m \choose \lceil m/(2c) \rceil} (np^2)^{\lceil m/(2c) \rceil} \le \exp(f_m),
$$
where 
\begin{eqnarray*}
f_m &=& k \left( \log \left( \frac {en}{k} \right) - \frac {p(k-1)}{2} \right) + m \left\{ \log \left[ \left( \frac {ek^2p}{2m} \right) \left(2ec \,n p^2 \right)^{1/(2c)} \right] + p \right\} \\
&=& k \left( \log (pn) + O(1) - \log \log (pn) - \frac {2+\eps'}{2} \log(pn) \right) \\
&& + \ \ m \left\{ \log \left[ \left( \frac {e(2+\eps')^2 \log^2 (pn)}{2mp} \right) \left( \frac {2ec(pn)^2}{n} \right)^{1/(2c)} \right] + p \right\}.
\end{eqnarray*}
Hence, $p_m \le \exp(-kg_m)$, where
\begin{eqnarray*}
g_m &=& \frac {\eps'}{2} \log(pn) - \frac{mp}{(2+\eps') \log(pn)} \left\{ \log \left[ \left( \frac {e(2+\eps')^2 \log^2 (pn)}{2mp} \right) \left( \frac {2ec(pn)^2}{n} \right)^{1/(2c)} \right] + p \right\}.
\end{eqnarray*}
Let 
$$
m' = \left( \frac {e(2+\eps')^2 \log^2 (pn)}{2p} \right) \left( \frac {2ecd^2}{n} \right)^{1/(2c)} = o \left( \frac {\log^2(pn)}{p} \right).
$$
First, note that if $m > em'$, then 
$$
g_m \ge \frac {\eps'}{2} \log(pn) - \frac{mp}{(2+\eps') \log(pn)} ( -1 + p ) \ge 1.
$$
For $m' \le m \le em'$, we get 
$$
g_m \ge \frac {\eps'}{2} \log(pn) - O \left( \frac{m'p^2}{\log(pn)} \right) = \frac {\eps'}{2} \log(pn) - o(1) \ge 1.
$$
Finally, for $m = m' / x$ for any $x = x(n) \ge 1$, since $\max_{x \ge 1} (\log x / x) = 1/e$, we get
\begin{eqnarray*}
g_m &\ge& \frac {\eps'}{2} \log(pn) - \frac{m'p/e}{(2+\eps') \log(pn)} - o(1) \\
&=& \frac {\eps'}{2} \log(pn) - \left( \frac {(2+\eps') \log (pn)}{2} \right) \left( \frac {2ec(pn)^2}{n} \right)^{1/(2c)} - o(1) \\
&=& \log(pn) \left[ \frac {\eps'}{2} - (1+o(1)) \left( \frac {2ec(pn)^2}{n} \right)^{1/(2c)} \right] - o(1) \ge 1.
\end{eqnarray*}
As a result, $p_m \le \exp(-k) = o(n^2)$ and so $\sum_{m=0}^{k \choose 2} p_m = o(1)$. The result holds.
\end{proof}

\bigskip

The last two lemmas give us easily the following lower bound on $\chi_c$ that matches (asymptotically) the upper bound. This proves part (a) of Theorem~\ref{thm:Gnp}. 
\begin{theorem}
Let $\omega$ be a function tending to infinity with $n$ arbitrarily slowly.
Suppose that $p=p(n)$ is such that 
$
\omega < pn \le n^{1/2-\omega/\sqrt{\log n}}.
$
Let $G = (V,E) \in \G(n,p)$. Then, \whp\ $\mcf(G) \leq (2+o(1)) \log (pn)/p$, and so
$$
\chi_c(G) \ge (1+o(1)) \frac {pn}{2 \log (pn)} \sim \chi(\G(n,p)).
$$
\end{theorem}
\begin{proof}
By~(\ref{eqn.sparselb}), we may assume that $pn - \log n \to \infty$ (or indeed $pn = \Omega(n^{1/3}/\omega)$ for some $\omega = \omega(n) \to \infty$). Let $k \sim 2 \log(pn)/p$ be defined as in Lemma~\ref{lem:lower_bound1}. As we aim for a statement that holds \whp, we may assume that all properties in Lemma~\ref{lem:lower_bound1} and Lemma~\ref{lem:triangles} hold deterministically. Consider any colouring that yields $\chi_c(G)$. We will show that no colour class has size at least $k$ which will finish the proof.  For a contradiction, suppose that the vertices of some set $K$ of size $k$ are of the same colour. It follows from Lemma~\ref{lem:lower_bound1}(a) that $G[K]$ contains at least one edge. As no maximal clique is monochromatic, each edge of $G[K]$ belongs to a triangle with the third vertex in $V \setminus K$. As, by Lemma~\ref{lem:triangles}, no edge belongs to more than $c$ triangles, we may greedily pick a set of edge-disjoint such triangles containing at least a $1/(2c)$-fraction of the edges in $G[K]$.  But this contradicts Lemma~\ref{lem:lower_bound1}(b). 
\end{proof}

\subsection{Lower bounds for parts (b) and (c)}

We start with the following observation. We will apply the result for $k=3$ but we state it in a more general case as the proof of this generalization is exactly the same.

\begin{lemma}\label{lem:lowerboundgeneral}
Let $k \geq 3$ be  a fixed integer, and let $\varepsilon \in (0,2/k)$ be a fixed arbitrarily small positive constant.  Suppose that $p=p(n)$ satisfies 
$$
(3k)^{ \frac{2}{k}} (\log n)^{\frac {2}{k(k-1)}} n^{1-\frac {2}{k}}  \leq pn \leq n^{1-\varepsilon},
$$
and let $s = s(n) := 3k p^{-k/2} \left(\log n \right)^{1/(k-1)}$.  Let $G=(V,E) \in \G(n,p)$. Then, \whp\ the following property holds: every set $S \subseteq V$ of $\lceil s \rceil$ vertices contains at least $\frac12 \binom{s}{k} p^{\binom{k}{2}} = \Theta\left(s^k p^{\binom{k}{2}}\right)$ copies of $K_k$.
\end{lemma}
\noindent (Note that the lower bound for $pn$ is there only to make sure that $s \le n$.)
\begin{proof}
For a fixed set $S \subseteq V$ of size $\lceil s \rceil$, let $X$ be the random variable counting the number of copies of $K_k$ in $S$. Clearly, $\mu = \E [X]=\binom{\lceil s \rceil}{k}p^{\binom{k}{2}}$. Using the notation of Lemma~\ref{thm:Janson}, we may deal separately with the cases $|A \cap B| = i$ with $2 \le i \le k$ to get
\begin{eqnarray*}
\overline{\Delta} &=& \sum_{i=2}^{k} \binom{\lceil s \rceil}{2k-i}\binom{2k-i}{k} \binom{k}{i} p^{2\binom{k}{2}-\binom{i}{2}} \sim \sum_{i=2}^k \frac {s^{2k-i}}{(k-i)!^2 i!} p^{k(k-1)-\binom{i}{2}}\\ 
&=& s^{2k} p^{k(k-1)} \sum_{i=2}^k \frac {s^{-i}}{(k-i)!^2 i!} p^{-\binom{i}{2}}.
\end{eqnarray*}
Let us observe that it follows from the definition of $s$ that the $i$-th term of the sum is of order $(1/p)^{-(i/2)(k-i+1)} (\log n)^{-i/(k-1)}$. Since $1/p \geq n^{\varepsilon}$, 
the logarithmic factor is negligible and it is easy to see that the sum is dominated by the last term ($i=k$). Therefore,
$$
\overline{\Delta} \sim \binom{\lceil s\rceil}{k}p^{\binom{k}{2}} = \mu.
$$
Applying Lemma~\ref{thm:Janson} with $t = \E [X]/2$ we get 
\begin{eqnarray*}
\Prob \left( X \leq \frac{\E [X]}{2} \right) &\le&  \exp \big( - (1+o(1)) \varphi(-1/2) \mu \big) \le \exp \left(-  \frac {0.15 s^k}{k!} p^{\binom{k}{2}} \right) \\
&=& \exp \left(- \frac {0.15 (3k)^{k-1} s \log n}{k!} \right) \le \exp( -2 s \log n).
\end{eqnarray*}
Taking a union bound over all $\binom{n}{\lceil s \rceil} \le \exp( \lceil s \rceil \log n)$ sets of size $\lceil s\rceil$, the desired property holds.
\end{proof}

We will also need the following property.
\begin{lemma}\label{lem:nodensegeneral}
Let $k \ge 3$ be a fixed integer, and let $\varepsilon = \eps(n)$ satisfy $(6/\log n)^{\frac{1}{k-1}} \le \eps \le 1/k$. 
Suppose $p=p(n)$ satisfies $pn \leq n^{k/(k+2)-\varepsilon}$, and let $G\in \G(n,p)$. Then \whp\ no edge in $G$ belongs to as many as $r = \lceil \eps^{-(k-1)} \rceil$ copies of $K_{k+1}$.
\end{lemma}
\begin{proof}
Let the graph $H$ be formed from an edge $e=uv$ together with some cliques $K_{k+1}$ containing $e$ (at least one such clique).  We will show that we have
\[ \frac{e(H)-1}{v(H)-2} \geq \frac{k+2}{2}.\]
 To see this, we may use induction on the number of cliques in a construction of $H$.  The inequality holds (at equality) if $H$ is a single $(k+1)$-clique.  Suppose that there are at least two cliques in the construction, and the last one added $n'$ new vertices and $e'$ new edges. If $n'=0$, then the inequality still holds after adding the last clique, so we may assume that $n' \geq 1.$ Then
\[ e' \ge n' (k\!+\!1\!-\!n') + {n' \choose 2} = n' \left(k + \frac 12 - \frac {n'}{2} \right) \geq  n' \left(k + \frac 12 - \frac {k-1}{2} \right) = n' \cdot  \frac{k+2}{2}, \]
where the second inequality follows since $n' \leq k-1$.  Thus $\frac{e'}{n'} \geq \frac{k+2}{2}$, and the induction step follows, since from $\frac{a+b}{c+d} \ge \min\{\frac{a}{c},\frac{b}{d}\}$ for positive $a,b,c$ and $d$, we may conclude that $\frac{e(H)-1}{v(H)-2} \geq \frac{k+2}{2}$.

The number of copies of $K_{k+1}$ on $2+\lfloor 1/\eps \rfloor$ vertices that contain the edge $uv$ is at most
\[ \binom{\lfloor 1/\eps \rfloor}{k-1} \leq \left(\frac{e (1/\eps)}{k-1}\right)^{k-1} < \eps^{-(k-1)}\]
for $k \geq 4$ (since $e<k-1$); and
$\binom{\lfloor 1/\eps \rfloor}{2} < \eps^{-2}$, so the result holds also for $k=3$.  
Thus if $H$ contains at least $r$ such cliques we must have $v(H) > 2+ 1/\eps$.  Further $H$ must then have a subgraph on at most $2+r(k-1)$ vertices which 
contains at least $r$ such cliques.
Call the set $W$ of vertices in $G$ \emph{dense} if 
\[ \frac{e(G[W])}{|W|-2} \geq \frac{k+2}{2}.\]
Let $a_i$ be the expected number of dense sets of size $i$ in $G$.
It will suffice for us to show that
\begin{equation} \label{eqn.showai}
  \sum_{ i=2+\lceil 1/\eps \rceil}^{2+r(k-1)} a_i = o(1).
\end{equation}
Now 
\begin{eqnarray*}
  a_{j+2} & \leq &  
    \binom{n}{j+2} \binom{\binom{j+2}{2}}{\lceil j(k+2)/2 \rceil} p^{\lceil j(k+2)/2 \rceil}\\
 & \leq &
  \left( \frac{en}{j+2} \right)^{j+2} \left( \frac{e (j+2)(j+1)}{j(k+2)} \cdot p \right)^{\lceil j(k+2)/2 \rceil}\\
 & \leq &
   n^2 \left( n p^{(k+2)/2} \right)^{j} \left( \frac{e}{j+2} \right)^{j+2}
   \left( \frac{e (j+2)(j+1)}{j(k+2)} \right)^{\lceil j(k+2)/2 \rceil}.
\end{eqnarray*}
But $n p^{(k+2)/2} \leq n^{-\eps(k+2)/2}$. Thus for $j \geq \lceil 1/\eps \rceil$, 
\[ n^2 \left( n p^{(k+2)/2} \right)^{j} \leq  n^2 \left( n^{- \eps (k+2)/2} \right)^{j}  \leq  n^2 n^{-2 \eps j} n^{- (\eps /2)j} \leq n^{- (\eps /2)j} \le n^{-1/2}, \]
as $k \geq 3$.  Since there are at most $2+r(k-1)=O(\eps^{-(k-1)})=o(n^{1/2})$ terms in the sum in~(\ref{eqn.showai}), the desired conclusion there follows.
\end{proof}
\medskip

Now, finally we are able to prove the lower bounds in parts (b) and (c) of Theorem~\ref{thm:Gnp}.

\begin{theorem}\label{thm:lower_bound_k}
Let $k \ge 3$ be a fixed integer, and let $\eps = \eps(n) = (6/\log n)^{\frac{1}{k-1}} = o(1)$. Suppose that $p=p(n)$ satisfies
$$
(3k)^{ \frac{2}{k}} (\log n)^{\frac {2}{k(k-1)}} n^{1-\frac {2}{k}}  \leq pn \leq n^{k/(k+2)-\varepsilon},
$$
and let $s = s(n) :=3k p^{-k/2} \left(\log n \right)^{1/(k-1)}$. Let $G=(V,E) \in \G(n,p)$. Then, \whp\ 
$\mcf(G) \leq s$ and so
$$
\chi_c(G) \ge \frac {n}{ s } = \Omega \left( \frac{(pn)^{k/2}n^{1-k/2}}{(\log n)^{1/(k-1)}} \right) = \Omega \left( \frac{np^{k/2}}{(\log n)^{1/(k-1)}} \right).
$$
\end{theorem}
\begin{proof}
For simplicity we shall ignore the fact that certain numbers should be integers: this never matters here.
Let $S$ be a fixed set of size $s$. 
We shall show that
$\pr(S \mbox{ \MCF})$ is very small.

If $S$ is \MCF, 
then it must be possible to extend each copy $J$ of $K_k$ inside $S$
to $K_{k+1}$ by adding some vertex $v \in V \setminus S$ which is complete to $J$, that is, $v$ is connected by an edge to each of the $k$ vertices of $J$. Unfortunately (for our proof) these edges can be reused in extensions for different copies of $K_k$ inside $S$. Our plan is to show that, if $S$ is \MCF, then with very high probability there must be a large collection of copies of $K_k$ inside $S$ that are extended to $V \setminus S$ using each edge between $S$ and $V \setminus S$ at most once; 
and this is very unlikely to happen.

Let $\mu= \binom{s}{k} p^{\binom{k}{2}}$, the expected number of copies of $K_k$ inside $S$.
Let $\mathcal{E}_S$ be the event that $S$ 
contains at least $\frac12 \mu$ copies of $K_k$. 
Let $\mathcal{F}$ be the event that each edge appears in 
at most $(1/\eps)^{k-1}$ copies of $K_{k+1}$. 
Let $\cF_S$ be the event that each edge $uv$ with $u \in S$ and $v \in V \setminus S$ appears in at most $(1/\eps)^{k-1}$ copies of $K_{k+1}$ such that each vertex other than $v$ is in $S$. 
Let $E_0$ be any set of edges within $S$ such that $\cE_S$ holds, and condition on $E|_S=E_0$.  Here $E|_{S}$ denotes the set of edges within $S$ in the random edge set $E$ of $G$.
 Observe that, since $E|_S$ is fixed, $\cF_S$ becomes a down-event in the lattice $\Omega$ of subsets of $E(S,V \setminus S)$.  
 %
Now let $t \geq 1$, and consider a possible `edge-disjoint extension of length $t$'.
Suppose that $J_1,\ldots,J_t$ are distinct copies of $K_k$ inside $S$, and $v_1,\ldots, v_t$ are vertices in $V \setminus S$ (not necessarily distinct), such that each possible edge appears at most once as $uv_i$ for some vertex $u$ in $J_i$. Then the event that $v_i$ is complete to $J_i$ for each $i \in [t]$ is the up-event in $\Omega$ that each of the $kt$ relevant edges is present.  Hence by Lemma~\ref{lem.Harris} (Harris's inequality)
\begin{equation}\label{eqn.disjext}
 \pr(v_i \mbox{ is complete to } J_i \mbox{ for each } i \in [t] \, | \, (E|_S=E_0) \land \cF_S) \leq p^{kt}.
\end{equation}

Keep the conditioning on $E|_S=E_0$, condition also on $\cF_S$, 
and suppose that $S$ is \MCF.  List the copies of $K_k$ inside $S$ in say lexicographic order, discarding any after the first $\frac12 \mu$; and consider the vertices in $V \setminus S$ in their natural order. 
Consider the first copy $J$ of $K_k$ on the list: test the vertices in $V \setminus S$ one by one until we find a vertex $v$ complete to $J$ (we must succeed since $S$ is maximal-clique-free).
Eliminate from the list the clique $J$ and all other copies $J'$ of $K_k$ such that $J'$ and $J$ have at least one vertex in common and $v$ is complete to $J'$ (so the extensions by $v$ have an edge in common).
Since we conditioned on $\cF_S$, at most $k(1/\eps)^{k-1}$ copies of $K_k$ are eliminated from the list in this round, and we move to the next copy of $K_k$ that is left in the list and continue. 
Since we also conditioned on $E|_S=E_0$, there are at least $\mu/2$ copies of $K_k$ in $S$,
and so the process must last for at least $t$ rounds, where
\[ t =  \frac {\mu/2}{k(1/\eps)^{k-1}}.\] 
Now we lower bound $t$: we have
\begin{eqnarray*}
t & = & \frac {1}{2 k(1/\eps)^{k-1}} \binom{s}{k} p^{\binom{k}{2}} ~~\sim~~ \frac {\eps^{k-1}}{2 k \cdot k!} s^k p^{\binom{k}{2}} \\
&\ge& \frac {\eps^{k-1}}{2 e k^2 (k/e)^k} \left( 3k p^{-k/2} (\log n)^{1/(k-1)} \right)^k p^{\binom{k}{2}}, 
\end{eqnarray*}
as $k! \le ek (k/e)^k$. We get 
$$
t \ge (1+o(1)) \frac {(e\eps)^{k-1} 3^k}{2 k^2} p^{-k/2} (\log n)^{k/(k-1)} > 3 \eps^{k-1} k p^{-k/2} (\log n)^{k/(k-1)} = \eps^{k-1} s \log n = 6 s,
$$
as $3^k e^{k-1}/(2k^3) > 3$ for $k \ge 3$. 

The number of sets of $t$ distinct copies $J_1,\ldots, J_t$ of $K_k$ inside $S$ which are on the list, and vertices $v_1,\ldots,v_t$ yielding possible edge-disjoint extensions as in~(\ref{eqn.disjext}), is at most $\binom{\frac12 \mu}{t} n^t$.  Thus, by~(\ref{eqn.disjext}) and a union bound,
\[ \pr(S \mbox{ \MCF} \, | \, (E|_S=E_0) \land \cF_S) \leq \binom{\frac12 \mu}{t} n^t p^{kt}.\]
Since this holds for every choice $E_0$ of edges within $S$ such that $\cE_S$ holds, and since $\cF \subseteq \cF_S$, it follows that
\[ \pr((S \mbox{ \MCF}) \land \cE_S \land \cF) \leq \binom{\frac12 \mu}{t} (n p^{k})^t.\]
We have
\begin{eqnarray*}
\binom{\frac12 \mu}{t}
(np^k)^t &\le & 
\exp \left( t \log \left( \frac{e(3e)^k(\log n)^{k/(k-1)}np^{k/2} }{2t}\right) \right) \\
&\le& \exp \left( t \log \left( \frac{e(3e)^k np^k \log n }{36k} \right) \right) \\
&\le& \exp \left( -t \left (\frac {2k}{k+2} + \eps k - 1+o(1) \right) \log n \right) \\ 
& \le & \exp \left( - \frac {t}{5+o(1)} \log n \right) 
\end{eqnarray*}
which is $o(n^{-s})$ since  $ t \ge (6+o(1))s$. 
Let us rename the fixed set $S$ as $S_0$. Then, with the union below being over all $s$-subsets $S$ of vertices,
\begin{eqnarray*}
&& \pr\left(\lor_{|S|=s} ((S \mbox{ \MCF}) \land \cE_{S} \land \cF) \right)\\
&\leq &\binom{n}{s}  \pr((S_0 \mbox{ \MCF}) \land \cE_{S_0} \land \cF)\\
& \leq & \binom{n}{s} \binom{\frac12 \mu}{t} (n p^{k})^t
 \; = \; o(1).
\end{eqnarray*}
Finally, let $\cE= \land_{|S|=s} \cE_S$.
Then 
\[ \pr(\mcf(G) \geq s) \leq \pr\left(\lor_{|S|=s} ((S \mbox{ \MCF}) \land \cE_S \land \cF) \right) + \pr(\bar{\cE})  +\pr(\bar{\cF}). \]
But each of the three terms in the upper bound here is $o(1)$.  We have just seen this for the first term: the second and third terms are $o(1)$ by
Lemmas~\ref{lem:lowerboundgeneral} and~\ref{lem:nodensegeneral} respectively. The proof of the theorem is complete.
\end{proof}


\subsection{Lower bounds for parts (d)-(g)}

Let us start with the following lemma, which is proved in a similar way to Lemma~\ref{lem:lowerboundgeneral}.

\begin{lemma}\label{lem:lowerboundgeneral2}
Let $k \ge 3$ be a fixed integer, and let $\varepsilon > 0$ be a fixed arbitrarily small positive constant. Let $C=C(n)$ be such that $1 \leq C=n^{o(1)}$. Suppose that $p=p(n)$ satisfies
$$
(3k)^{2/(k+2)} (\log n)^{\frac{2}{(k-1)(k+2)}} n^{\frac{k}{k+2}} C^{-2/(k+2)} \leq pn \leq n^{1-\varepsilon}.
$$
Let $s = s(n)=Cpn$, and let $G=(V,E) \in \G(n,p)$. Then, \whp\ the following property holds: every set $S \subseteq V$ of $\lceil s \rceil$ vertices contains at least $\frac12 \binom{\lceil s \rceil}{k} p^{\binom{k}{2}} =\Theta\left( s^k p^{\binom{k}{2}}\right)$ copies of $K_k$.
\end{lemma}
\begin{proof}
As before, for simplicity we shall ignore the fact that certain numbers should be integers.
For a fixed set $S \subseteq V$ of size $\lceil s \rceil$, let $X$ be the random variable counting the number of copies of $K_k$ in $S$. Clearly, $\mu = \E [X]=\binom{\lceil s \rceil}{k}p^{\binom{k}{2}}$. Using the notation of Lemma~\ref{thm:Janson}, we may independently deal with the cases $|A \cap B| = i$ with $2 \le i \le k$ to get
\begin{eqnarray}\label{eqarray:Delta} 
\nonumber \overline{\Delta} &=& \sum_{i=2}^{k} \binom{\lceil s \rceil}{2k-i}\binom{2k-i}{k} \binom{k}{i} p^{2\binom{k}{2}-\binom{i}{2}} \sim \sum_{i=2}^k \frac {s^{2k-i}}{(k-i)!^2 i!} p^{k(k-1)-\binom{i}{2}}\\ 
&=& s^{2k} p^{k(k-1)} \sum_{i=2}^k \frac {s^{-i}}{(k-i)!^2 i!} p^{-\binom{i}{2}}.
\end{eqnarray}
For $i=2,3,\ldots, k-1$, let $a_i=s^{-i} p^{-\binom{i}{2}}$ and note that the $i$th term in the last sum is $\Theta(a_i)$.
The ratio $a_{i+1}/a_i = s^{-1}p^{-i} = \frac{1}{Cn}p^{-i-1}$,
which is increasing with $i$.  It follows that $a_2,\ldots,a_k$ is a unimodal sequence, with maximum either $a_2$ or $a_k$.
Moreover, since $1/p \ge n^{\varepsilon}$, at most two consecutive terms can be of the same order. As a result, the sum is of order the larger of the term $i=2$ and the term $i=k$. 
(Note that we did not rule our the possibility that the $(k-1)$-st term is of order of the $k$-th term yet.)
More precisely, comparing the two terms, the sum is of order of the term $i=k$ if $p \ll n^{-2/(k+3)} C^{-2/(k+3)}$, and of order of the term $i=2$ otherwise.  Moreover, in the case $p \ll n^{-2/(k+3)} C^{-2/(k+3)}$, by comparing the $(k-1)$-st and the $k$-th term, we see that the $k$-th term dominates the $(k-1)$-st term if $p \ll (Cn)^{-1/k}$, and since $n^{-2/(k+3)} C^{-2/(k+3)} \le (Cn)^{-1/k}$ for $k \ge 3$, this condition is satisfied. Hence, if $p \ll n^{-2/(k+3)} C^{-2/(k+3)}$, the sum is in fact asymptotic to the term $i=k$. 

Suppose first that $p \ll n^{-2/(k+3)} C^{-2/(k+3)}$. As we already mentioned, in this case $\overline{\Delta} \sim \mu$, and so applying Lemma~\ref{thm:Janson} with $t = \E [X]/2$ we get 
\begin{eqnarray*}
\Prob \left( X \leq \frac{\E [X]}{2} \right) &\le&  \exp \big( - (1+o(1)) \varphi(-1/2) \mu \big) \le \exp \left(-  \frac {0.15 (Cpn)^k}{k!} p^{\binom{k}{2}} \right) \\
&\leq&   \exp \left(- \frac {0.15 (3k)^{k-1} (Cpn) \log n}{k!} \right) \le \exp( -2 s \log n),
\end{eqnarray*}
by our assumption that $pn \ge (3k)^{2/(k+2)} (\log n)^{\frac{2}{(k-1)(k+2)}} n^{\frac{k}{k+2}} C^{-2/(k+2)}$.
Taking a union bound over all $\binom{n}{\lceil s \rceil} \le \exp( \lceil s \rceil \log n)$ sets of size $\lceil s\rceil$, the desired property holds.

Assume now that $p=\Omega(n^{-2/(k+3)} C^{-2/(k+3)})$. In this case we get 
$$
\overline{\Delta} = \Theta \left(\mu^2 p^{-1} s^{-2}\right).
$$
Therefore, applying Lemma~\ref{thm:Janson} with $t = \E [X]/2$ we get 
\begin{eqnarray*}
\Prob \left( X \leq \frac{\E [X]}{2} \right) &\le&  \exp \big( - \Omega(\mu^2/\overline{\Delta}) \big) = \exp \left(-  \Omega \left( \frac{s C (pn)^2}{n}\right)\right) \le \exp( -2 s \log n), 
\end{eqnarray*}
as, by our assumption on $p$ and $C$, we have $C(pn)^2/n \ge n^{2k/(k+2)-1+o(1)} \ge n^{1/5+o(1)} \gg \log n$.
Taking a union bound over all $\binom{n}{\lceil s \rceil} \le \exp( \lceil s \rceil \log n)$ sets of size $\lceil s\rceil$, the desired property holds.
\end{proof}

We can now prove the following theorem that implies the lower bounds in parts (e) and (g) of Theorem~\ref{thm:Gnp}.

\begin{theorem}\label{thm:lower_bound_k2}
Let $k \ge 3$ be a fixed integer. Suppose that $p=p(n)$ satisfies 
$$
(\log n)^{\frac{2k}{(k-1)(k+2)}} n^{\frac{k}{k+2}} \leq pn \leq (4e^2 k \log^2 n)^{-1/k} n^{(k-1)/k}.
$$
Let $s = s(n) = Cpn$, where 
\begin{equation}\label{eq:def_C}
C=C(n)=\left(\frac {32 k^2 \log n}{-\log( p^k n 4ek\log^2 n)} \right)^{1/(k-1)}.
\end{equation}
Let $G=(V,E) \in \G(n,p)$. Then, \whp\  $\mcf(G) \leq s$, and so
$$
\chi_c(G) \ge \frac {n}{s} = \frac {1}{Cp} = \frac{n^{o(1)}}{p}.
$$
\end{theorem}

\noindent 
Before we move to the proof of this result, let us make a few comments. First, obviously, if there are several values of $k$ that satisfy the assumptions on $pn$, one should consider the one that gives the best lower bound for $\chi_c(G)$. In particular, the case $k=3$ covers values of $pn$ between $n^{3/5} (\log n)^{3/5}$ and $n^{2/3} (12e^2 \log^2n)^{-1/3}$, and the case $k=4$ covers the range of $pn$ between $n^{2/3} (\log n)^{4/9}$ and $n^{3/4} (16e^2\log^2n)^{-1/4}$. The case $k=5$ can be applied already if $pn$ is at least $n^{5/7} (\log n)^{5/14}$, and so for values of $pn$ in the interval $[n^{5/7} (\log n)^{5/14}, n^{3/4} (16e^2\log^2n)^{-1/4}]$ both $k=4$ and $k=5$ satisfy the conditions of the theorem. Similarly, if we fix $\eps >0$ (arbitrarily small), all values of $pn$ that belong to the interval $[n^{5/7} (\log n)^{5/14}, n^{1-\eps}]$ are covered by at least two values of $k$ at most the constant $2/\eps$ (since $p \geq n^{-\frac{2}{k+2}}$, for $p \leq n^{-\eps}$ we are concerned only with $k$ such that $n^{-\frac{2}{k+2}} \leq n^{-\eps}$, that is $k \leq 2/\eps -2 \leq 2/\eps$).
If $p$ is in the interval for $k-1$ then
$p \leq n^{-1/(k-1)} \leq n^{-1/k - \eta}$, for some
$\eta \geq 1/(k-1)-1/k=1/k(k-1) 
\geq \eps^2/4$.
Thus for at least one of the relevant values $k$ (any one except the smallest) we have $pn \le n^{(k-1)/k - \eta}$, and so $C \le (32k/(\eta+o(1)))^{1/(k-1)} = O(1)$.
This in turn implies that $\mcf(G) =O(pn)$ and so $\chi_c(G)=\Omega(1/p)$. The only range not covered by this theorem is when $pn = n^{3/5+o(1)}$ or $pn = n^{2/3+o(1)}$ which will be done separately later on.

\begin{proof}
%
The proof is similar to the proof of Theorem~\ref{thm:lower_bound_k}. However, one additional idea is needed: informally speaking, we now have too many copies of $K_k$, but still only $\Theta(s)$ of them are going to be extended to $K_{k+1}$ using disjoint edge sets. The union bound over all choices of these would be too large to be successfully applied, and so we focus on a randomly chosen subset of the $K_k$'s. To this end, define 
\[ L:= \left\lceil \frac{n^{k\!-\!1} p^{(k+2)(k-1)/2}}{(k-1)! (\log n)^2} \right \rceil \ge \frac {(\log n)^{k-2}}{(k\!-\!1)!}.\]
Now, in order to obtain the desired smaller subset of cliques, we do the following, independently of the random graph $G$:
for each set of $k$ vertices, we independently colour it \emph{red} with probability $1/L$. Our first goal is to show that no edge $uv$ belongs to too many cliques of size $k+1$ which after removal of $v$ form a red $k$-set. 

Let $\eta=\eta(n) = o(1)$, say $\eta=n^{-\tfrac1{11}}$.
For each pair of distinct vertices $u$ and $v$, let $\mathcal{E}_{u,v}$ be the event that 
\begin{equation}\label{eq:codeg}
 \left| |N(u) \cap N(v)| - np^2 \right| \leq \eta \, np^2.
\end{equation} 
Let $\cE$ be the event that $\cE_{u,v}$ holds for each pair $u \neq v$.  Since $\eta^2 np^2 \geq n^{1/55}$, by the Chernoff bounds~\eqref{eq:chern1} and~\eqref{eq:chern2}, $\pr(\bar{\cE}_{u,v}) =o(n^{-2})$; and so by a union bound, \whp\ $\cE$ holds.

Fix any ordered pair of distinct vertices $u, v$, and any set $I \subseteq V \setminus \{u,v\}$ with $|I|=k-1$. Let $X_{u,v,I}=1$ if $I$ induces a clique, $u$ and $v$ are both complete to $I$ (we do not require $u$ and $v$ to be adjacent). Finally, let $X_{u,v}=\sum_{I \subseteq V \setminus \{u,v\}} X_{u,v,I}$.
Let $m=m(n) = \lfloor (1+\eta)np^2 \rfloor$.  Let $Z$ be the number of $(k-1)$-cliques in the random graph with vertex set $[m]$ and edge-probability $p(n)$.   Clearly, conditional on $\cE_{u,v}$,$\, X_{u,v}$ is stochastically at most $Z$.  We have
\begin{equation}\label{eq:expectation_for_Vu}
\mathbb{E} Z=\binom{m}{k-1} p^{\binom{k-1}{2}} \sim L\log^2 n \geq \frac {\log^k n}{(k-1)!} \gg (\log n)^{k-2}.
\end{equation}
Hence, by Lemma~\ref{lem:vu}, once $n$ is sufficiently large that $\mathbb{E}Z \leq (4/3) L \log^2 n$,
\begin{eqnarray*}
\mathbb{P}(X_{u,v} \geq 2 L \log^2 n \, | \, \cE_{u,v}) &\leq& \mathbb{P}(Z \geq (3/2) \mathbb{E}Z) \\
&\leq& \exp\Big(- \Omega( (\mathbb{E}Z)^{1/(k-2)}) \Big)\\
&=& \exp\Big(-\omega (\log n)\Big) = o(n^{-2}).
\end{eqnarray*}
Thus
\[ \pr(X_{u,v} \geq 2 L \log^2 n) \leq \pr(X_{u,v} \geq 2 L \log^2 n \, | \, \cE_{u,v})  + \pr(\bar{\cE}_{u,v}) = o(n^{-2}).\]
By taking a union bound over all ordered pairs of distinct vertices, we see that \whp\ for each $u \neq v$ we have $X_{u,v} \le 2 L \log^2 n$. Now let $Y_{u,v,I}$ be defined as $X_{u,v,I}$ with the additional condition that $I \cup \{u\}$ is red, and let $Y_{u,v}=\sum_I Y_{u,v,I} \le X_{u,v}$. Clearly, conditional on $X_{u,v} \le 2 L \log^2 n$, $Y_{u,v}$ is stochastically bounded above by the binomial random variable $\Bin(2 L \log^2 n, 1/L)$ with expectation $2 \log^2 n$. It is a straightforward application of~\eqref{eq:chern2} (together with a union bound over all ordered pairs $u,v$) to see that \whp\ for each pair $u\neq v$ we have $Y_{u,v} \le 4 \log^2 n$. Let $\mathcal{D}$ be the event that this property holds, so that $\cD$ holds \whp.  
This completes the first part of the proof.
\smallskip

For the rest of the proof, we argue very much as in the proof of Theorem~\ref{thm:lower_bound_k}.
As before, we ignore rounding issues for $s$, since this does not matter.
Let  $s=Cpn$ and let $S$ be a set of size $s$.  
Let $\cD_S$ be the event that each edge $uv$ with $u \in S$ and $v \in V \setminus S$ appears in at most $4 \log^2 n$ copies of $K_{k+1}$ which are such that the $k$ vertices other than $v$ form a red $k$-set in $S$.  Note that $\cD \subseteq \cD_S$.
Let $\mu= \binom{s}{k} p^{\binom{k}{2}}$, the expected number of $k$-cliques in $S$; so $\mu/L$ is the expected number of red $k$-cliques in $S$.
Let $\cH_S$ be the event that $S$ contains at least $\mu/(4L)$ 
red cliques of size $k$; and let $\cH = \land_{|S'|=s} \cH_{S'}$.  (We shall check later that $\cH$ holds \whp) 

Denote the random colouring of the $k$-sets of vertices by $R$; and let $R|_S$ denote its restriction to the $k$-sets in $S$.  Let $E_0$ be any set of edges within $S$ and let $R_0$ be any colouring of the $k$-sets within $S$ such that $\cH_S$ holds when $E|_S=E_0$ and $R|_S=R_0$.  Condition on  $E|_S=E_0$ and $R|_S=R_0$.
Let $t \geq 1$, and consider a possible `edge-disjoint extension of length $t$'.
Suppose that $J_1,\ldots,J_t$ are distinct red cliques of size $k$ inside $S$, and $v_1,\ldots, v_t$ are vertices in $V \setminus S$ (not necessarily distinct), such that each possible edge appears at most once as $uv_i$ for some vertex $u$ in $J_i$. Then, as in~(\ref{eqn.disjext}),\begin{equation} \label{eqn.disjext2}
 \pr(v_i \mbox{ is complete to } J_i \mbox{ for each } i \in [t] \, | \, (E|_S=E_0) \land (R|_S=R_0) \land \cD_S) \leq p^{kt}.
\end{equation}
Also, as before, if 
$S$ is \MCF, then by considering a list of red $k$-cliques of length $\mu/4L$, there must be a disjoint extension of length  $t:=(\mu/4L)/(4k \log^2 n)$.  
Hence, by a union bound,
\[ \pr(S \mbox{ \MCF} \, | \, (E|_S=E_0) \land (R|_S=R_0) \land \cD_S) \leq \binom{\mu/4L}{t} n^t p^{kt}.\]
Since this holds for each choice of $E_0$ and $R_0$ such that $\cH_S$ holds,
\[ \pr(S \mbox{ \MCF} \mid \cH_S \land \cD_S) \leq \binom{\mu/4L}{t}(np^k)^t.\]
Note that
\begin{eqnarray*} 
\binom{\mu/4L}{t}(np^k)^t &\le& \left( 4 e k np^k \log^2 n \right)^t \\
&=& \exp\left( \frac{\mu}{16 kL \log^2 n} \log (4eknp^k  \log^2 n)\right)\\
&=& \exp\left((1+o(1))\frac{sC^{k-1}}{16k^2}\log (4eknp^k \log^2 n)\right).
\end{eqnarray*}
Plugging in the value of $C$, we see that this expectation is at most $\exp(-(2+o(1))s \log n)$. 
Hence
\[ \binom{n}{ s } \pr \left(S \mbox{ \MCF} \land \cH_S \land \cD_S \right) = o(1).\]
But, letting $\cH=\land_{|S|=s} \cH_S$,  
\[ \pr(\mcf(G) \geq s) \leq \binom{n}{ s } \pr \left(S \mbox{ \MCF} \land \cH_S \land \cD_S \right) + \pr(\bar{\cH}) + \pr(\bar{\cD}).\]
%
We have just seen that the first term in the upper bound is $o(1)$, and we noted earlier that $\cD$ holds \whp; so it remains to show that $\pr(\bar{\cH})=o(1)$.

Let $S_0$ be a fixed set of $s$ vertices, and let $S$ now be an arbitrary such set.
Let $\mathcal{G}_S$ be the event that $S$ contains at least $\frac12 \mu=\frac12 \binom{ s }{k} p^{\binom{k}{2}}$ copies of $K_k$; and let $\cG= \land_{|S|=s} \cG_S$.  By Lemma~\ref{lem:lowerboundgeneral2} we have $\pr(\bar{\cG})=o(1)$.
Also, by the Chernoff bound~\eqref{eq:chern1},
\[ \pr(\bar{\cH}_S \mid \cG_S) = \exp(-\Omega(\mu/L))= \exp(-\Omega(s \log^2 n));  \]
and so, by a union bound,
\[ \pr \left( \lor_{|S|=s} (\bar{\cH}_S \land \cG_S) \right) \leq \binom{n}{s} 
\pr(\bar{\cH}_{S_0} \mid \cG_{S_0}) = o(1).\]
But
\[ \pr(\bar{\cH})= \pr(\lor_{|S|=s} \bar{\cH}_S) \leq \pr \left( \lor_{|S|=s} (\bar{\cH}_S \land \cG_S) \right) + \pr(\bar{\cG}).\]
and so $\pr(\bar{\cH})=o(1)$, as required.  This completes the proof of the theorem.
%
%
\end{proof}

Finally, we deal with the missing gaps when $pn = n^{3/5+o(1)}$ or $pn = n^{2/3+o(1)}$ which will finish the lower bounds in parts (d) and (f) of Theorem~\ref{thm:Gnp}.
\begin{theorem}
Suppose that $p=p(n)$ satisfies 
\begin{itemize}
\item [(a)] $n^{3/5-6/(\log n)^{1/2}} \leq pn \leq n^{3/5} (\log n)^{3/5}$, or
\item [(b)] $pn = n^{2/3+o(1)}$ and $pn\leq n^{2/3} (\log n)^{4/9}$.
\end{itemize}
Let $G=(V,E) \in \G(n,p)$. Then, \whp, 
$$
\chi_c(G) \ge \frac{n^{o(1)}}{p}.
$$
\end{theorem}

\begin{proof}[Sketch of the proof]
Let us focus on part (a) first. Since the proof of this theorem is almost identical to the one of Theorem~\ref{thm:lower_bound_k2} (with $k=3$), we only mention two technical issues that are relatively easy to deal with. First, note that in order to apply Lemma~\ref{lem:lowerboundgeneral2}, we cannot keep $C$ as defined in~(\ref{eq:def_C}); this time, we set
$$
C = \max \left( \left(\frac {32 k^2 \log n}{-\log( p^k n 4ek\log^2 n)} \right)^{1/(k-1)}\!\!, \, 3k (\log n)^{1/(k-1)} (pn)^{-(k+2)/2} n^{k/2} \right).
$$
We do not necessarily have $C=O((\log n)^{1/2})$ anymore but still it is the case that $1 \le C = n^{o(1)}$. (Recall that $k=3$ in the proof of part (a).) This time there is no need to reduce the number of cliques for the union bound to work, so we may keep $L=1$, which simplifies the argument slightly. The last difference is with the application of Lemma~\ref{lem:vu}. This time the expected value of $X_{u,v} = X_{u,v}(p)$ (see~(\ref{eq:expectation_for_Vu})) is not large enough for the lemma to be directly applied. However, as standard in such situations, one can increase the probability $p$ to some value $p' \ge p$ for which the expected value of $X_{u,v}(p')$ is of order greater than $(\log n)^{k-2}$ (which equals $\log n$ in the case $k=3$ we deal with in part(a)) as required. For such a value $p'$, it follows that \whp\  for any pair $u,v$ we have $X_{u,v}(p') \le 2 \log^2 n$,
 and by standard coupling arguments the same holds for $p$. The rest of the proof is not affected. 

Exactly the same adjustments are required for part (b), this time with $k=4$. 
\end{proof}

\section{Concluding Remarks}
Let us pick up two points for further thought.
\smallskip 
\begin{itemize}
\item We investigated the clique colouring number $\chi_c(G)$ for random graphs $G \in \G(n,p)$, and in Theorem~\ref{thm:Gnp} we obtained fairly good estimates for values of $p=p(n)$, other than $p = n^{-\frac12 +o(1)}$ where $\chi_c$ drops dramatically as $p$ increases.  By parts (a) and (c), for suitable $\eps(n)=o(1)$ (going to 0 slowly), if $p = n^{-\frac12 - \eps}$ then $\chi_c(G) = n^{\frac12 - o(1)}$ \whp; whereas if $p = n^{-\frac12 + \eps}$ then $\chi_c(G) = n^{\frac14 + o(1)}$ \whp. 
 For intermediate values of $p$, all we say in part (b) is that $\chi_c(G)$ lies \whp between the values $n^{\frac14 - o(1)}$ and $n^{\frac12 + o(1)}$.  It would be interesting to learn more about this jump. 
\item  A second natural question for random graphs $G \in \G(n,p)$ concerns the dense case, when $p$ is a constant with $0<p<1$.  We have seen that $\chi_c(G) $ is $O(\log n)$ \whp\ but what about a lower bound?
\end{itemize}


\begin{thebibliography}{99}

\bibitem{Andreae} T.~Andreae, M.~Schughart, Zs.~Tuza, Clique-transversal sets of line graphs and complements of line graphs, \emph{Discrete Math.} \textbf{88}, 1991, 11--
20.

\bibitem{BGGPS2004}
G.~Bacs{\'o}, S.~Gravier, A.~Gy{\'a}rf{\'a}s, M.~Preissmann, A.~Seb\H{o}, Coloring the maximal cliques of graphs,
  \emph{SIAM Journal on Discrete Mathematics} {\bf 17}, 2004, 361--376.

\bibitem{Campos} C.N.~Campos, S.~Dantas, C.P.~de Mello, Colouring clique-hypergraphs of circulant graphs. \emph{Electron. Notes Discret. Math.} \textbf{30}, 2008, 189--194.

\bibitem{Cerioli} M.R.~Cerioli, A.L.~Korenchendler, Clique-coloring Circular-Arc graphs. \emph{Electron. Notes Discret. Math.} \textbf{35}, 2009, 287--292.
\bibitem{Defossez} D.~D\'{e}fossez, Clique-coloring some classes of odd-hole-free graphs, \emph{J. Graph Theory} \textbf{53}, 2006, 233--249.

\bibitem{DMNPPT} R.M.~del Rio-Chanona, C.~MacRury, J.~Nicolaidis, X.~Perez-Gimenez, P.~Pra\l{}at, and K.~Ternovsky, Injective colouring of binomial random graphs, preprint.

\bibitem{DeMarcoKahn} R.~DeMarco, J.~Kahn, Tight upper-tail bounds for cliques, \emph{Random Structures \& Algorithms} \textbf{41}(4), 2012, 469--487.

\bibitem{DMP} A.~Dudek, D.~Mitsche, and P.~Pra\l{}at, The set chromatic number of random graphs, \emph{Discrete Applied Mathematics} \textbf{215}, 2016, 61--70.



\bibitem{Gravier} S.~Gravier, C.~Ho\`{a}ng, F.~Maffray, Coloring the hypergraph of maximal cliques of a graph with no long path, \emph{Discrete Math.}, \textbf{272}, 2003, 285--290.


\bibitem{FMPP} A.~Frieze, D.~Mitsche, X.~P\'erez-Gim\'enez, P.~Pra\l{}at, On-line list colouring of random graphs, \emph{Electronic Journal of Combinatorics} \textbf{22}(2), 2015, \#P2.41.

\bibitem{Harris} T.E.~Harris, A lower bound for the critical probability in a certain percolation, \emph{Proceedings of the Cambridge Philosophical Society}, \textbf{56}, 1960, 13--20.

\bibitem{JLR} S.~Janson, T.~{\L}uczak, A.~Ruci\'nski, \emph{Random Graphs}, Wiley, New York, 2000.

\bibitem{JRDeletion} S.\ Janson, A.\ Ruci\'nski, The deletion method for upper tail estimates, \emph{Combinatorica}, \textbf{24} (4) (2004), 615--640.


\bibitem{KM15} R.J.~Kang, C.~McDiarmid, Colouring random graphs, in L. Beineke and R.J. Wilson, eds, \emph{Topics in Chromatic Graph Theory}, 2015.
   
\bibitem{Klein} S.~Klein, A.~Morgana, On clique-colouring with few $P_4$'s, \emph{J. Braz. Comput. Soc.} \textbf{18}, 2012, 113--119.

\bibitem{Kratochvil} J.~Kratochv\'{i}l, Zs.~Tuza, On the complexity of bicoloring clique hypergraphs of graphs, \emph{J. Algorithms}, \textbf{45}, 2002, 40--54.
\bibitem{KSVW02} M.~Krivelevich, B.~Sudakov, V.H.~Vu, N.C.~Wormald, On the probability of independent sets in random graphs,  \emph{Random Structures \& Algorithms} \textbf{22} (1), 2003, 1--14.

\bibitem{Liang} Z.~Liang, E.~Shan, L~ Kang, Clique-coloring claw-free graphs, \emph{Graphs and Combinatorics}, DOI10.1007/s00373-015-1657-8, 2015, 1--16. 
\bibitem{L} T.~ \L{}uczak, The chromatic number of random graphs, \emph{Combinatorica} \textbf{11}(1), 1991, 45--54. 

\bibitem{LMU03}
M.~Luczak, C.~McDiarmid, E.~Upfal,
On-line routing of random calls in networks,
{\em Probability Theory and Related Fields} {\bf 125}, 2003, 457--482.

\bibitem{M} C.J.H.~McDiarmid, On the chromatic number of random graphs, \emph{Random Structures \& Algorithms} \textbf{1}(4), 1990, 435--442. 

\bibitem{OurpaperGnr} C.J.H.~McDiarmid, D.~Mitsche, P.~Pra\l{}at, Clique colouring of geometric graphs, in preparation.


\bibitem{Mohar} B.~Mohar, R.~Skrekovski, The Gr\"{o}tzsch Theorem for the hypergraph of maximal cliques. \emph{Electronic J. Comb.} \textbf{6}, 1999,  R26.



\bibitem{SLK2014} E.~Shan, Z.~Liang, L.~Kang,
Clique-transversal sets and clique-coloring in planar graphs,
\emph{European J. Combin.} {\bf 36}, 2014, 367--376.

\bibitem{Vu1} V.H.~Vu, On some degree conditions which guarantee the upper bound of chromatic (choice) number of random graphs, \emph{Journal of Graph Theory}, \textbf{31}, 1999, 201--226.

\bibitem{Vu} 
V.~Vu, A large deviation result on the number of small subgraphs of a random graph, {\em Combinatorics, Probability and Computing} \textbf{10}(1), 2001, 79--94. 


\end{thebibliography}
\end{document}